\documentclass{amsart}
\usepackage{amssymb,mathtools}
\usepackage[british]{babel}
\usepackage{enumitem}
\usepackage[latin1]{inputenc}

\usepackage{url}

\newtheorem{theorem}{Theorem}
\numberwithin{theorem}{section}
\newtheorem{corollary}[theorem]{Corollary}
\newtheorem{lemma}[theorem]{Lemma}
\newtheorem{proposition}[theorem]{Proposition}
\theoremstyle{definition}
\newtheorem{definition}[theorem]{Definition}
\newtheorem{remark}[theorem]{Remark}

\numberwithin{equation}{section}

\newcommand{\bqo}{\mathsf{BQO}}
\newcommand{\delbqo}{\Delta^0_2\text{-}\mathsf{BQO}}
\newcommand{\aca}{\mathsf{ACA}}
\newcommand{\rca}{\mathsf{RCA}}
\newcommand{\zfc}{\mathsf{ZFC}}
\newcommand{\hth}{\operatorname{ht}}

\newcommand{\len}{\operatorname{lh}}
\newcommand{\tr}{\operatorname{Tr}}
\newcommand{\supp}{\operatorname{supp}}
\newcommand{\tl}{\vartriangleleft}
\newcommand{\three}{\mathbf 3}
\newcommand{\one}{\mathbf 1}
\newcommand{\tc}{\operatorname{TC}(\Omega\oplus\Omega)}
\newcommand{\tcleq}{\leq_{\operatorname{TC}}}
\newcommand{\tcle}{<_{\operatorname{TC}}}

\makeatletter
\newcommand{\doverline}[1]{\overline{\dbl@overline{#1}}}
\newcommand{\dbl@overline}[1]{\mathpalette\dbl@@overline{#1}}
\newcommand{\dbl@@overline}[2]{%
  \begingroup
  \sbox\z@{$\m@th#1\overline{#2}$}%
  \ht\z@=\dimexpr\ht\z@-2\dbl@adjust{#1}\relax
  \box\z@
  \ifx#1\scriptstyle\kern-\scriptspace\else
  \ifx#1\scriptscriptstyle\kern-\scriptspace\fi\fi
  \endgroup
}
\newcommand{\dbl@adjust}[1]{%
  \fontdimen8
  \ifx#1\displaystyle\textfont\else
  \ifx#1\textstyle\textfont\else
  \ifx#1\scriptstyle\scriptfont\else
  \scriptscriptfont\fi\fi\fi 3
}
\makeatother

\title[On the better quasi order with three elements]{On the logical strength of\\ the better quasi order with three elements}
\author{Anton Freund}

\address{Anton Freund, Department of Mathematics, Technical University of Darmstadt, Schloss\-garten\-str.~7, 64289~Darmstadt, Germany}
\email{freund@mathematik.tu-darmstadt.de}

\thanks{Funded by the Deutsche Forschungsgemeinschaft (DFG, German Research Foundation) -- Project number 460597863.}

\begin{document}

\begin{abstract}
The notion of better quasi order ($\mathsf{BQO}$), due to Nash-Williams, is very fruitful mathematically and intriguing from the standpoint of logic, due to several long-standing open problems. In the present paper, we make a significant step towards one of these: Let $\mathbf 3$ be the discrete order with three~elements. We show that arithmetical recursion along the natural numbers ($\mathsf{ACA}_0^+$) follows from $\mathbf 3$ being $\mathsf{BQO}$, over the base theory $\mathsf{RCA_0}$ from reverse mathematics. Also over the latter, we deduce arithmetical transfinite recursion ($\mathsf{ATR}_0$) from the assumption that $\mathbf 3$ is $\Delta^0_2\text{-}\mathsf{BQO}$, which plays a role in work of Montalb\'an.
\end{abstract}

\keywords{Better quasi order, reverse mathematics, 3 is bqo}

\maketitle

\section{Introduction}\label{sect:intro}

A quasi order~$Q$ is called a well quasi order if any infinite sequence $f:\mathbb N\to Q$ admits $i<j$ with $f(i)\leq_Q f(j)$. Well quasi orders are central in combinatorics and have important applications in computer science, e.\,g., in the context of the graph minor theorem due to N.~Robertson and P.~Seymour~\cite{robertson-seymour-polytime,robertson-seymour-gm}. At the same time, the class of well quasi orders lacks closure properties under infinitary constructions. For example, when $Q$ is a well quasi order, the same need not hold for sequences in~$Q$ that are indexed by ordinals (with a natural order relation). To restore closure under this and further operations, C.~Nash-Williams~\cite{nash-williams-trees,nash-williams-bqo} has introduced the notion of better quasi order ($\bqo$). A famous application is R.~Laver's~\cite{laver71} proof of R.~Fra\"iss\'e's conjecture that the scattered linear orders are well quasi ordered by embeddability. For the definition of $\bqo$s and associated notation that is not explained in our paper, we refer to a survey by A.~Marcone~\cite{marcone-survey-new} (see, e.\,g., \cite{kruskal-rediscovered,pequignot-survey}~for further background on well and better quasi orders). To give an idea of the shape of the definition, we recall that a quasi order~$Q$ is $\bqo$ if any array $f:B\to Q$ for a barrier~$B\subseteq[\mathbb N]^{<\omega}$ admits $s\tl t$ with $f(s)\leq_Q f(t)$. The special case where $(B,\tl)$ coincides with $([\mathbb N]^1,\tl)\cong(\mathbb N,<)$ yields the weaker notion of well quasi order.

From the perspective of mathematical logic, many questions about better quasi orders remain wide open. In particular, we do not know which axioms are needed to prove that the discrete order $\three$ with underlying set~$\{0,1,2\}$ is~$\bqo$. Within the framework of reverse mathematics (see, e.\,g., the textbook by S.~Simpson~\cite{simpson09}), it is known that the rather strong axiom system $\mathsf{ATR_0}$ (arithmetical transfinite recursion) can prove that $\three$ is $\bqo$. Both A.~Marcone and A.~Montalb\'an have asked whether any weaker axiom system suffices (see Questions~3.3 and~25 of \cite{marcone-survey-old} and~\cite{montalban-open-problems}, respectively). The possibility that strong axioms may be needed to prove a property of a small finite object seems quite intriguing. It gains additional relevance from Montalb\'an's~\cite{montalban-fraisse} analysis of Fra\"iss\'e's conjecture, which shows that the latter reduces to $\three$ being $\Delta^0_2\text{-}\bqo$, over $\mathsf{ATR}_0$. In view of this fact, it is a major open question whether $\mathsf{ATR}_0$ proves that~$\three$ is $\Delta^0_2\text{-}\bqo$. We recall that the latter is provable in $\Pi^1_1\text{-}\mathsf{CA}_0$, since this stronger theory proves that the notions of $\bqo$ and of $\Delta^0_2\text{-}\bqo$ are equivalent (as shown by Montalb\'an~\cite{montalban-fraisse}).

In the present paper, we show that $\three$ being~$\bqo$ entails arithmetical recursion along~$\mathbb N$ (the principal axiom of~$\mathsf{ACA_0^+}$), over the usual base theory~$\rca_0$ from reverse mathematics. With the same base theory, we also prove that~$\three$ being $\Delta^0_2\text{-}\bqo$ implies the stronger principle of arithmetical transfinite recursion. As far as the author is aware, no lower bounds for the logical strength of these statements have been known so far. We do not know whether either bound is sharp. Indeed, the previous paragraph suggests that this may be hard to determine.

According to the standard view of reverse mathematics, our result shows that any proof that $\three$ is $\bqo$ must make use of complex infinite sets. It does also improve a result on Higman's lemma for $\bqo$s~\cite{freund-higman-bqo}. As a side product of our proof, we learn that any finite quasi order is $\bqo$ if the same holds~for~$\three$, still over~$\rca_0$. Here the quantification over finite quasi orders is internal. The corres\-ponding result with externally fixed orders is due to Marcone~\cite{marcone-survey-old}.

As indicated above, we work with `barrier-$\bqo$s' rather than `block-$\bqo$s', which appears to be the standard (cf.~Definition~2.3 of~\cite{montalban-fraisse} and the paragraph that follows it). In any case, it is immediate that any block-$\bqo$ is a barrier-$\bqo$. To account for the converse direction, we will first show that $\three$ being barrier-$\bqo$ implies arithmetical comprehension ($\aca_0$). The latter entails weak K\H{o}nig's lemma, which ensures that the notions of barrier-$\bqo$ and block-$\bqo$ are equivalent, as shown by P.~Cholak, A.~Marcone and R.~Solomon~\cite{cholak-RM-wpo}. This means that the choice of definition does not affect our main results after all.

To give an idea of our proofs, we indicate how a property of the finite object~$\three$ can lead to complex infinite sets. It is known that a quasi order~$Q$ is $\bqo$ precisely when the same holds for~$H_{\aleph_1}(Q)$, the collection of hereditarily countable sets with urelements from~$Q$, equipped with a suitable order. We refer to~\cite{forster-bqo,pequignot-survey} for proofs of this result, which goes back to Nash-Williams~\cite{nash-williams-trees}. In a somewhat similar vein, Montalb\'an~\cite{montalban-fraisse} has shown that a certain order~$(\tr(Q),\leq^w)$ of infinite $Q$-labelled trees is $\bqo$ when $Q$ is $\Delta^0_2\text{-}\bqo$, provably in~$\mathsf{ATR}_0$. The point for us is that $H_{\aleph_1}(\three)$ and $\tr(\three)$ are infinite even though $\three$ is finite.

We will show that the results on $H_{\aleph_1}(Q)$ and $\tr(Q)$ become available over~$\mathsf{RCA_0}$ when we restrict to the collection $H_f(Q)$ of hereditarily finite sets and to the set~$\tr_f(Q)$ of finite trees, respectively. This can be seen as a generalization of previous work of Marcone~\cite{marcone-survey-old} and the present author~\cite{freund-higman-bqo}, which is concerned with finitely iterated powersets. Together with known results, it already follows that $\three$ being $\Delta^0_2\text{-}\bqo$ is unprovable in~$\aca_0^+$. Indeed, the proof theoretic ordinal~$\varphi_2(0)$ of this theory admits an order repflecting map into~$(\tr(\mathbf 2),\leq^s)$ and then into $(\tr(\three),\leq^w)$, as shown by Marcone and Montalb\'an~\cite{marcone-montalban-hausdorff,montalban-fraisse} (where $\leq^s$ is a stricter order than~$\leq^w$).

The statement that some computable order is well founded has complexity~$\Pi^1_1$, so that it cannot imply the typical $\Pi^1_2$-principles of reverse mathematics. To overcome this obstacle, one can consider well ordering principles, which assert that some computable transformation~$F$ of countable linear orders preserves well foundedness. It is known that $F(\alpha)=\omega^\alpha$, $F(\alpha)=\varepsilon_\alpha$ and $F(\alpha)=\varphi_\alpha(0)$ correspond to arithmetical comprehension, arithmetical recursion along~$\mathbb N$ and arithmetical transfinite recursion, respectively~\cite{rathjen-afshari,friedman-montalban-weiermann,girard87,hirst94,marcone-montalban,rathjen-weiermann-atr}.

When we try to derive that some~$F$ is a well ordering principle, we face the challenge that $H_f(\three)$ and $\tr_f(\three)$ will not embed $F(\alpha)$ for every countable well order~$\alpha$. To resolve this, we take inspiration from the full structure~$H_{\aleph_1}(\three)$ after all. Consider the order $\Omega\oplus\Omega$ with two independent copies of a countable well order~$\Omega$. The idea is to model the elements of each copy by von Neumann ordinals over urelements~$\{0,1\}$ and~$\{1,2\}$, to show that $\Omega\oplus\Omega$ embeds into~$H_{\aleph_1}(\three)$. Officially, we will use elements of $\Omega\oplus\Omega$ as names for their images in~$H_{\aleph_1}(\three)$, so that the argument can be carried out in~$\rca_0$. Over the latter, we will learn that $\Omega\oplus\Omega$ is~$\bqo$ if the same holds for~$\three$. In order to show that~$F$ is a well ordering principle, it will then be enough to construct an order reflecting map from each value $F(\alpha)$ into $H_f(\Omega\oplus\Omega)$ or $T_f(\Omega\oplus\Omega)$, for a suitable~$\Omega$ that depends on~$\alpha$.

Concerning the previous paragraph, we note that two urelements are required for each copy of the von Neumann ordinals, as the order over a single urelement is trivial (see Remark~\ref{rmk:singleton-supports}). Indeed, it must be necessary to use three urelements in total, since $\mathbf 2$ being $\bqo$ can already be proved in~$\rca_0$, as shown by Marcone~\cite{marcone-survey-old}.

\section{From better quasi orders to arithmetical comprehension}\label{sect:ACA}

In the first part of this section, we consider the collection~$H_{\aleph_1}(Q)$ of hereditarily countable sets with urelements from a given quasi order~$Q$. The base theory is inessential, as our observations on~$H_{\aleph_1}(Q)$ play no official role in any of the subsequent arguments, even though they provide crucial motivation. For definiteness, we agree to work in Zermelo-Fraenkel set theory with choice ($\zfc$). In the second part of the section, we work in~$\rca_0$ and mimic an argument that involves~$H_{\aleph_1}(\three)$, to show that arithmetical comprehension follows from $\three$ being $\bqo$.

\begin{definition}[$\zfc$]\label{def:V(Q)}
Given a quasi order~$(Q,\leq_Q)$, the class $H_{\aleph_1}(Q)$ is generated by the following recursive clauses:
\begin{enumerate}[label=(\roman*)]
\item we have $\langle 0,q\rangle\in H_{\aleph_1}(Q)$ for each $q\in Q$,
\item each countable $a\subseteq H_{\aleph_1}(Q)$ yields another element $\langle 1,a\rangle\in H_{\aleph_1}(Q)$.
\end{enumerate}
The quasi order $\leq_{H(Q)}$ on $H_{\aleph_1}(Q)$ is recursively explained as follows:
\begin{align*}
\langle 0,p\rangle\leq_{H(Q)}\langle 0,q\rangle\quad&\Leftrightarrow\quad p\leq_Q q,\\
\langle 0,p\rangle\leq_{H(Q)}\langle 1,b\rangle\quad&\Leftrightarrow\quad \langle 0,p\rangle\leq_{H(Q)}y\text{ for some }y\in b,\\
\langle 1,a\rangle\leq_{H(Q)}\langle 0,q\rangle\quad&\Leftrightarrow\quad x\leq_{H(Q)}\langle 0,q\rangle\text{ for all }x\in a,\\
\langle 1,a\rangle\leq_{H(Q)}\langle 1,b\rangle\quad&\Leftrightarrow\quad \text{for each $x\in a$ there is $y\in b$ with $x\leq_{H(Q)}y$}.
\end{align*}
With each $a\in H_{\aleph_1}(Q)$ we associate a support set $\supp(a)\subseteq Q$ that is given by
\begin{equation*}
\supp(\langle 0,q\rangle)=\{q\}\quad\text{and}\quad \supp(\langle 1,a\rangle)=\bigcup\{\supp(x)\,|\,x\in a\}.
\end{equation*}
In order to improve readability, we will from now on write $q\in Q$ and $a\subseteq H_{\aleph_1}(Q)$ to refer to $\langle 0,q\rangle$ and $\langle 1,a\rangle$, respectively.
\end{definition}

One readily checks the following by induction on ranks.

\begin{lemma}[$\zfc$]\label{lem:V(Q)-basic}
For any quasi order~$Q$, we have the following:
\begin{enumerate}[label=(\alph*)]
\item Given $x\leq_{H(Q)}y$, we get $\supp(x)\leq_{H(Q)}\supp(y)$.
\item For $x\in a\subseteq H_{\aleph_1}(Q)$ we have $x\leq_{H(Q)}a$.
\end{enumerate}
\end{lemma}

It is known that $Q$ is $\bqo$ precisely if $H_{\aleph_1}(Q)$ is $\bqo$ or, which is equivalent in this case, if $H_{\aleph_1}(Q)$ is well founded (see, e.\,g.,~\cite{pequignot-survey}). In Section~\ref{sect:ACA-plus}, we will prove the direction from left to right for the hereditarily finite part of~$H_{\aleph_1}(Q)$, with base theory~$\rca_0$. The following explains why we will mostly work with sets $a\subseteq H_{\aleph_1}(Q)$ for which $\supp(a)$ has at least two elements.

\begin{remark}\label{rmk:singleton-supports}
A straightforward induction shows that, for any $q\in Q$, we have
\begin{equation*}
\supp(x)\leq_{H(Q)} q\leq_{H(Q)}\supp(y)\quad\Rightarrow\quad x\leq_{H(Q)}y.
\end{equation*}
This means that any set $a\subseteq H_{\aleph_1}(Q)$ with $\supp(a)=\{q\}$ is equivalent to~$q$ in the order on~$H_{\aleph_1}(Q)$, and that the sets with empty support form an initial segment. One can show that this initial segment is linear, as the order between two sets with empty support coincides with the order between their ranks.
\end{remark}

Recall that $\three$ denotes the discrete order on~$\{0,1,2\}$. We now show that $H_{\aleph_1}(\three)$ embeds two incomparable copies of the von Neumann ordinals.

\begin{definition}[$\zfc$]
For all countable ordinals $\alpha$ we define $\dot\alpha,\ddot\alpha\in H_{\aleph_1}(\three)$ recursively by $\dot\alpha=\{0,1\}\cup\{\dot\gamma\,|\,\gamma<\alpha\}$ and $\ddot\alpha=\{1,2\}\cup\{\ddot\gamma\,|\,\gamma<\alpha\}$.
\end{definition}

As promised, we have the following.

\begin{proposition}[$\zfc$]\label{prop:embed-Omega-plus-Omega}
When $\alpha$ and $\beta$ are countable ordinals, we have
\begin{equation*}
\dot\alpha\leq_{H(\three)}\dot\beta\quad\Leftrightarrow\quad\alpha\leq\beta\quad\Leftrightarrow\quad\ddot\alpha\leq_{H(\three)}\ddot\beta.
\end{equation*}
Furthermore, $\dot\alpha$ and $\ddot\beta$ are incomparable in $H_{\aleph_1}(\three)$.
\end{proposition}
\begin{proof}
We inductively get $\supp(\dot\alpha)=\{0,1\}$ and $\supp(\ddot\beta)=\{1,2\}$. Incomparability follows by Lemma~\ref{lem:V(Q)-basic}. The latter also ensures that $\dot\alpha\leq\dot\beta$ follows from $\dot\alpha\in\dot\beta$ and hence from $\alpha<\beta$. To show the converse by contradiction, we assume that $\beta$ is minimal such that $\dot\alpha\leq\dot\beta$ holds for some $\alpha>\beta$. In view of $\dot\beta\in\dot\alpha$ we must have $\dot\beta\leq y$ for some $y\in\dot\beta$. Note that $y\in\{0,1\}$ is impossible since the two elements of~$\supp(\dot\beta)=\{0,1\}$ are incomparable. We must thus have $y=\dot\gamma$ for some~$\gamma<\beta$, against minimality. The other equivalence holds for the same reasons.
\end{proof}

For quasi orders $Q_0$ and $Q_1$, we define $Q_0\oplus Q_1$ as the quasi order given by
\begin{gather*}
Q_0\oplus Q_1:=\{\langle i,q\rangle\,|\, i\in\{0,1\}\text{ and }q\in Q_i\},\\
\langle i,p\rangle\leq_{Q_0\oplus Q_1}\langle j,q\rangle\quad:\Leftrightarrow\quad i=j\text{ and }p\leq_{Q_i}q.
\end{gather*}
The following combines different results due to Marcone. We provide an integrated proof, amongst others because it will help to motivate subsequent generalizations. Note that $\omega$ denotes $\mathbb N$ with the usual order.

\begin{proposition}[$\rca_0$; \cite{marcone-survey-old}]\label{prop:aca-from-oplus}
Arithmetical comprehension follows from the statement that all arrays $[\mathbb N]^3\to\omega\oplus\alpha$ for any well order~$\alpha$ are good.
\end{proposition}
\begin{proof}
Given a quasi order~$Q$, we define a quasi order on the collection $P_f(Q)$ of finite subsets of~$Q$ by stipulating
\begin{equation*}
a\leq_{P(Q)}b\quad\Leftrightarrow\quad\text{for each $x\in a$ there is $y\in b$ with $x\leq_Q y$}.
\end{equation*}
Assume that $f:[\mathbb N]^n\to P_f(Q)$ is a bad array. For each $s=\langle s_0,\ldots,s_n\rangle\in[\mathbb N]^{n+1}$ we may then pick $g(s)\in f(\langle s_0,\ldots,s_{n-1}\rangle)$ with $g(s)\not\leq_Q q$ for all $q\in f(\langle s_1,\ldots,s_n\rangle)$. This yields a bad array $g:[\mathbb N]^{n+1}\to Q$. Given that all arrays $[\mathbb N]^3\to\omega\oplus\alpha$ are good, we thus learn that $P_f(P_f(\omega\oplus\alpha))$ is a well partial order. Now consider
\begin{equation*}
\omega^\alpha:=\{\langle\alpha_0,\ldots,\alpha_{n-1}\rangle\,|\,\alpha_i\in\alpha\text{ for all $i<n$ and }\alpha_{n-1}\leq\ldots\leq\alpha_0\}.
\end{equation*}
We get a linear order~$\preceq$ on $\omega^\alpha$ by stipulating
\begin{equation*}
\langle\alpha_0,\ldots,\alpha_{m-1}\rangle\preceq\langle\beta_0,\ldots,\beta_{n-1}\rangle\,\Leftrightarrow\,\begin{cases}
\text{either $\alpha_i=\beta_i$ for all $i<m\leq n$},\\[1ex]
\parbox{.47\textwidth}{or there is $j<\min(m,n)$ with $\alpha_j<\beta_j$\\ and $\alpha_i=\beta_i$ for all $i<j$.}
\end{cases}
\end{equation*}
Over $\rca_0$, arithmetical comprehension is equivalent to the statement that $\omega^\alpha$ is well founded for any well order~$\alpha$, as shown by J.-Y.~Girard~\cite{girard87} and J.~Hirst~\cite{hirst94}. To complete the proof, it is thus enough to construct an order reflecting map
\begin{equation*}
h:\omega^\alpha\to P_f(P_f(\omega\oplus\alpha)).
\end{equation*}
For $n\in\omega$ and $\gamma\in\alpha$, we abbreviate $\langle n,\gamma\rangle:=\{(0,n),(1,\gamma)\}\in P_f(\omega\oplus\alpha)$. Since the summands of $\omega\oplus\alpha$ are incomparable by definition, we see that $\langle m,\gamma\rangle\leq\langle n,\delta\rangle$ is equivalent to the conjunction of~$m\leq n$ and $\gamma\leq\delta$. We now put
\begin{equation*}
h(\langle\alpha_0,\ldots,\alpha_{n-1}\rangle)=\{\langle i,\alpha_i\rangle\,|\,i<n\}.
\end{equation*}
Let us consider $\alpha=\langle\alpha_0,\ldots,\alpha_{m-1}\rangle$ and $\beta=\langle\beta_0,\ldots,\beta_{n-1}\rangle$ with $h(\alpha)\leq h(\beta)$. For each $i<m$ we find a $j<n$ with $\langle i,\alpha_i\rangle\leq\langle j,\beta_j\rangle$. In view of $\beta_{n-1}\leq\ldots\leq\beta_0$ we can conclude $i\leq j$ and then $\alpha_i\leq\beta_j\leq\beta_i$. As $i<m\leq n$ was arbitrary, this yields the desired inequality $\alpha\preceq\beta$.
\end{proof}

Over $\rca_0$, we shall now show that $\Omega\oplus\Omega$ is $\bqo$ if the same holds for~$\three$, for an arbitrary well order~$\Omega$. If we choose the latter so that it embeds both $\omega$ and a given ordinal~$\alpha$, we can derive arithmetical comprehension by the previous proposition. The first step is to define a set~$\tc$ that represents the transitive closure of $\Omega\oplus\Omega\subseteq H_{\aleph_1}(\three)$ via the embedding from Proposition~\ref{prop:embed-Omega-plus-Omega}. From the viewpoint of the theory~$\rca_0$, the elements of $\tc$ are names rather than actual sets.

\begin{definition}[$\rca_0$]
Given a well order~$\Omega$, we consider the disjoint union
\begin{equation*}
\tc=\three\cup\{\overline\alpha\,|\,\alpha\in\Omega\}\cup\{\doverline\alpha\,|\,\alpha\in\Omega\}.
\end{equation*}
Let $\tcleq$ be the partial order on $\tc$ in which the only strict inequalities are
\begin{gather*}
i\tcle\overline\alpha\text{ for }i\in\{0,1\},\quad i\tcle\doverline\alpha\text{ for }i\in\{1,2\},\\
\overline\alpha\tcle\overline\beta\text{ and }\doverline\alpha\tcle\doverline\beta\text{ for }\alpha<\beta.
\end{gather*}
For $a\in\tc$ we define $E(a)\subseteq\tc$ by setting $E(i)=\{i\}$ for $i\in\three$ as well as $E(\overline\alpha)=\{0,1\}\cup\{\overline\gamma\,|\,\gamma<\alpha\}$ and $E(\doverline\alpha)=\{1,2\}\cup\{\doverline\gamma\,|\,\gamma<\alpha\}$.
\end{definition}

In view of the set theoretic considerations above, it might have been more natural to set $E(i)=\emptyset$ for $i\in\three$. However, the given definition allows for a particularly concise statement of the following result.

\begin{lemma}[$\rca_0$]\label{lem:order-Omega-plus-Omega}
Given $a,b\in\tc$, we have
\begin{equation*}
a\tcleq b\quad\Leftrightarrow\quad\text{for each $x\in E(a)$ there is $y\in E(b)$ with $x\tcleq y$}.
\end{equation*}
The direction from right to left is witnessed by a map $\frak x:\tc^2\to\tc$ such that $a\not\tcleq b$ entails both $\frak x(a,b)\in E(a)$ and $\frak x(a,b)\not\tcleq y$ for all $y\in E(b)$.
\end{lemma}
\begin{proof}
For the direction from left to right, we observe that $a\tcleq b$ implies (and will thus be equivalent to) the apparently stronger statement $E(a)\subseteq E(b)$. To witness the converse direction, we set $\frak x(a,b)=a$ for $a\in\three$ as well as
\begin{gather*}
\frak x(\overline\alpha,1)=0,\quad \frak x(\overline\alpha,i)=1\text{ for }i\in\{0,2\},\quad\frak x(\overline\alpha,\overline\beta)=\overline\beta,\quad\frak x(\overline\alpha,\doverline\beta)=0,\\
\frak x(\doverline\alpha,1)=2,\quad \frak x(\overline\alpha,i)=1\text{ for }i\in\{0,2\},\quad\frak x(\doverline\alpha,\doverline\beta)=\doverline\beta,\quad\frak x(\doverline\alpha,\overline\beta)=2.
\end{gather*}
The desired property can be verified explicitly.
\end{proof}

Let us write $1+\Omega$ for the extension of~$\Omega$ by a new smallest element, which will be denoted by~$\bot$. In the set theoretic discussion above, the notion of rank has been fundamental. Here we have the following analogue.

\begin{lemma}[$\rca_0$]\label{lem:E-ordinal-descent}
There is a function $o:\tc\to1+\Omega$ such that $x\in E(a)$ with $a\notin\three$ implies $o(x)<o(a)$.
\end{lemma}
\begin{proof}
Clearly the claim holds when $o(i)=\bot$ for $i\in\three$ and $o(\overline\alpha)=o(\doverline\alpha)=\alpha$.
\end{proof}

The following result will be superseded by Theorem~\ref{thm:Omega+Omega_bqo} below. It is convenient to prove the weaker result first, because this will secure arithmetical comprehension, which facilitates the rest of the argument. In particular, the upper index in $[\mathbb N]^3$ can be replaced by any $n\in\mathbb N$. We have chosen $n=3$ in view of Proposition~\ref{prop:aca-from-oplus}.

\begin{proposition}[$\rca_0$]\label{prop:Omega-plus-Omega-for-ACA}
If $\three$ is $\bqo$, then all arrays $[\mathbb N]^3\to\Omega\oplus\Omega$ for any well order~$\Omega$ are good.
\end{proposition}
\begin{proof}
Clearly $\Omega\oplus\Omega$ embeds into~$\tc$. Aiming at the contrapositive, we assume that $f:[\mathbb N]^3\to\tc$ is bad. For $s=\langle s_0,\ldots,s_n\rangle\in[\mathbb N]^{n+1}$~with~$n>0$, we set $s^0:=\langle s_0,\ldots,s_{n-1}\rangle$ and $s^1:=\langle s_1,\ldots,s_n\rangle$. Note that we have $s^0\tl s^1$ and that $s\tl t$ entails $s^1=t^0$ when $s,t\in[\mathbb N]^{n+1}$ have the same length. With~$\frak x$ as given by Lemma~\ref{lem:order-Omega-plus-Omega}, we define $g:\mathcal N=\bigcup_{n\geq 3}[\mathbb N]^n\to\tc$ recursively by
\begin{equation*}
g(s)=\begin{cases}
f(s) & \text{when $s\in[\mathbb N]^3$},\\
\mathfrak x(g(s^0),g(s^1)) & \text{otherwise}.
\end{cases}
\end{equation*}
By induction on~$n\geq 3$, we show that $s\tl t$ entails $g(s)\not\tcleq g(t)$ when $s,t\in[\mathbb N]^n$ have the same length. For $n=3$ this holds by the assumption on~$f$. Now consider the induction step for $s\tl t$ in~$[\mathbb N]^{n+1}$. Since we have $s^0\tl s^1=t^0$, the induction hypothesis and Lemma~\ref{lem:order-Omega-plus-Omega} yield $g(s)\not\tcleq y$ for all $y\in E(g(t^0))$. Due to the same lemma we obtain $g(t)\in E(g(t^0))$, so that we indeed have $g(s)\not\tcleq g(t)$. Let us also observe that we get $o(g(t))<o(g(t^0))$ in case we have $g(t^0)\notin\three$, due to Lemma~\ref{lem:E-ordinal-descent}. For $g(t^0)\in\three$ we get $g(t)=g(t^0)$ because of $E(g(t^0))=\{g(t^0)\}$. Let us now set
\begin{equation*}
B=\left\{\left.t\in\mathcal N\,\right|\,g(t)\in\three\text{ but }g(s)\notin\three\text{ for all }s\sqsubset t\text{ in }\mathcal N\right\},
\end{equation*}
where $s\sqsubset t$ denotes that $s$ is a proper initial segment of~$t$. To establish that $B$ is a block with base $\mathbb N$, we must show that each infinite set $X\subseteq\mathbb N$ admits an $n\geq 3$ with~$X[n]\in B$. Here $X[n]$ is the set of the $n$ smallest elements of~$X$, identified with its increasing enumeration. We will similarly write $t[n]$ when $t$ is a finite sequence of length at least~$n$. Let us note that we have $X[n+1]^0=X[n]$.  If there was no $n\geq 3$ with $X[n]\in B$, the above would thus yield $o(X[3])>o(X[4])>\ldots$, against the well foundedness of~$\Omega$. We~will construct a barrier~$B^*$ such that each $t\in B^*$ has an initial segment~$t[n]\in B$. In view of $g(t[n])\in\three$ and $t[n+1]^0=t[n]$, the above will ensure $g(t[n])=g(t[n+1])=\ldots=g(t)$, which means that $g$ restricts to an array~$B^*\to\three$. To see that this array must be bad, we consider $s,t\in B^*$ with~$s\tl t$. For sufficiently large $n\geq 3$, we find $s',t'\in[\mathbb N]^n$ with $s\sqsubset s'$ and $t\sqsubset t'$ as well as~$s'\tl t'$. The point is that $s'$ and $t'$ have the same length, so that the above yields~$g(s')\not\tcleq g(t')$. Due to $g(s)\in\three$ and $s\sqsubset s'$ we have $g(s)=g(s')$ and similarly $g(t)=g(t')$, as we have just seen. Hence we indeed get $g(s)\not\tcleq g(t)$. It remains to construct the barrier $B^*$. The following is due to Cholak, Marcone and Solomon~\cite{cholak-RM-wpo}, who use weak K\H{o}nig's lemma to cover a somewhat more general situation. We reproduce their argument in order to show that $\rca_0$ suffices in the present case. For $s,t\in[\mathbb N]^{<\omega}$ we write $s\ll t$ to denote that $s=\langle s_0,\ldots,s_{n-1}\rangle$ and $t=\langle t_0,\ldots,t_{n-1}\rangle$ have the same length and validate $s_i\leq t_i$ for all $i<n$. Let us observe that each $t$ will only admit finitely many~$s\ll t$. Following~\cite{cholak-RM-wpo}, we consider
\begingroup
\allowdisplaybreaks
\begin{align*}
T={}&\{s\in[\mathbb N]^{<\omega}\,|\,r\notin B\text{ for all }r\sqsubset s\},\\
T^*={}&\{s\in[\mathbb N]^{<\omega}\,|\,r\in T\text{ for some }r\ll s\},\\
B^*={}&\{s\in T^*\,|\,t\notin T^*\text{ for all }t\in[\mathbb N]^{<\omega}\text{ with }s\sqsubset t\}={}\\*
{}&\{s\in T^*\,|\,r\in B\text{ for all }r\ll s\text{ with }r\in T\}.
\end{align*}%
\endgroup
Note that $T$ and $T^*$ are trees with leaves~$B$ and $B^*$. Reflexivity of~$\ll$ yields $T\subseteq T^*$, which entails that elements of~$B^*$ have initial segments in~$B$, as promised above. The equality between the two characterizations of~$B^*$ is verified in~\cite[Section~5]{cholak-RM-wpo}. It ensures that $B^*$ can be formed in~$\rca_0$. Given that $B$ is a block with base~$\mathbb N$, the tree $T$ has no infinite path. The crucial task, for which \cite{cholak-RM-wpo} uses weak K\H{o}nig's lemma, is to show that the same holds for~$T^*$. Towards a contradiction, we assume that $h:\mathbb N\to\mathbb N$ validates $h[n]=\langle h(0),\ldots,h(n-1)\rangle\in T^*$ for all~$n\in\mathbb N$. Let us put
\begin{equation*}
\overline o(n):=\max\{o(g(s))\,|\,s\in T\text{ and }s\ll h[n]\}
\end{equation*}
for $n\geq 3$. We note that the maximum is taken over a set that is finite but nonempty, due to the definition of~$T^*$. To get a contradiction with the well foundedness of~$\Omega$, we now show $\overline o(n+1)<\overline o(n)$. Pick $s\in T$ with $s\ll h[n+1]$ and $\overline o(n+1)=o(g(s))$. For $r=s^0$ we have $r\ll h[n]$ as well as $r\sqsubset s$ and hence~$r\in T$. We get $g(s)\in E(g(r))$ as above. In view of $r\sqsubset s\in T$ we must have $g(r)\notin\three$. By invoking Lemma~\ref{lem:E-ordinal-descent} once again, we can conclude $\overline o(n+1)=o(g(s))<o(g(r))\leq\overline o(n)$. As in the proof of Lemma~5.11 from~\cite{cholak-RM-wpo}, we now learn that $B^*$ is a barrier: Given an infinite $X\subseteq\mathbb N$, pick $n\in\mathbb N$ with $X[n]\in T^*$ but $X[n+1]\notin T^*$. Aiming at a contradiction, we assume $X[n]\notin B^*$. By the second characterization of $B^*$ above, we find $r\ll X[n]$ with $r\in T$ and $r\notin B$. Pick $r'\ll X[n+1]$ with $r\sqsubset r'$. We then have $r'\in T$ and hence $X[n+1]\in T^*$, against our assumption. This shows that $B^*$ is a block with base~$\mathbb N$. To see that it is a barrier, we consider $s,t\in B^*$. Towards a contradiction, we assume $s\subset t$. Then $t$ is strictly longer than $s$. For~$t'\sqsubset t$ with the same length as~$s$, we get $t'\ll s$. Due to $t\in T^*$ we find $r\ll t$ with~$r\in T$. Let $r'\sqsubset r$ have the same length as $s$. We obtain $r'\in T$ and $r'\ll t'\ll s$. By the second characterization of~$B^*$ we get $r'\in B$, which contradicts $r\in T$. To summarize, we have shown that a bad array $[\mathbb N]^3\to\Omega\oplus\Omega$ yields a bad array $B^*\to\three$ that is defined on a barrier, as needed to establish the contrapositive.
\end{proof}

The following result will be superseded by Corollary~\ref{cor:three-to-acaplus}.

\begin{corollary}[$\rca_0$]\label{cor:three-aca}
Arithmetical comprehension follows from $\three$ being $\bqo$.
\end{corollary}
\begin{proof}
For a given well order $\alpha$, we consider the well order
\begin{equation*}
\Omega:=\omega+\alpha=\{\langle 0,n\rangle\,|\,n\in\omega\}\cup\{\langle 1,\gamma\rangle\,|\,\gamma\in\alpha\}
\end{equation*}
in which the first summand is placed below the second. More explicitly, this means that we have $\langle 0,n\rangle<\langle 1,\gamma\rangle$ for any $n\in\omega$ and $\gamma\in\alpha$, as well as $\langle 0,m\rangle<\langle 0,n\rangle$ for~$m<n$ and $\langle 1,\beta\rangle<\langle 1,\gamma\rangle$ for $\beta<\gamma$. The point is that $\omega\oplus\alpha$ can be embedded into $\Omega\oplus\Omega$. Given that $\three$ is $\bqo$, the previous proposition ensures that all arrays $[\mathbb N]^3\to\omega\oplus\alpha$ are good. Proposition~\ref{prop:aca-from-oplus} yields arithmetical comprehension.
\end{proof}

In the presence of weak K\H{o}nig's lemma and in particular of arithmetical comprehension, the definitions of $\bqo$s in terms of barriers and blocks are equivalent (see~\cite{cholak-RM-wpo}). This helps to prove the following strengthening of Proposition~\ref{prop:Omega-plus-Omega-for-ACA}.

\begin{theorem}[$\rca_0$]\label{thm:Omega+Omega_bqo}
If $\three$ is $\bqo$, then so is $\Omega\oplus\Omega$ for every well order~$\Omega$.
\end{theorem}
\begin{proof}
Aiming at the contrapositive, we assume that $f:B\to\tc$ is a bad array for some barrier~$B$. Let us consider $\mathcal B=\bigcup_{n\geq 1}B^n$ with
\begin{equation*}
B^n=\{s(0)\cup\ldots\cup s(n-1)\,|\,s(0)\tl\ldots\tl s(n-1)\text{ and }s(i)\in B\text{ for }i<n\}.
\end{equation*}
The idea is that $\mathcal B$ replaces $\mathcal N$ from the proof of Proposition~\ref{prop:Omega-plus-Omega-for-ACA}. Each $s\in\mathcal B$ can be uniquely written as $s=s(0)\cup\ldots\cup s(n-1)$ with $s(i)\in B$ and $s(0)\tl\ldots\tl s(n-1)$. To see this, use induction on~$i$ to verify that $s(i)\cup\ldots\cup s(n-1)$ is $s$ without its smallest $i$ elements. Next, note that $s(i)$ is an initial segment of $s(i)\cup\ldots\cup s(n-1)$. There is only one such $s(i)$ in the block~$B$. Since $B$ is a barrier, the maximal element of $s(i+1)$ will always exceed the one of~$s(i)$. This shows that $n$ is unique and bounded by the length of~$s$. The latter entails that $s\in\mathcal B$ is decidable. For~$s$ as above, we now set $s^0:=s(0)\cup\ldots\cup s(n-2)$ and $s^1:=s(1)\cup\ldots\cup s(n-1)$, provided that we have $n>1$. The above entails that $s^0$ is an initial segment of~$s$, which yields $s^0\tl s^1$. Given $s\tl t$ with $s,t\in B^n$ for a single~$n$, we get $s^1=t^0$, also by the above. We now define $g:\mathcal B\to\tc$ by
\begin{equation*}
g(s)=\begin{cases}
f(s) & \text{if }s\in B=B^1,\\
\mathfrak x(g(s^0),g(s^1)) & \text{otherwise}.
\end{cases}
\end{equation*}
As in the proof of Proposition~\ref{prop:Omega-plus-Omega-for-ACA}, one can employ induction on $n$ to show that $s\tl t$ entails $g(s)\not\tcleq g(t)$ when we have $s,t\in B^n$ for a single~$n$. Once again we get $g(t)\in E(g(t^0))$, which yields $o(g(t))<o(g(t^0))$ when $g(t^0)\notin\three\subseteq\tc$ and $g(t)=g(t^0)$ when $g(t^0)\in\three$. Let us now consider
\begin{equation*}
B'=\{t\in\mathcal B\,|\,g(t)\in\three\text{ but }g(s)\notin\three\text{ for all }s\sqsubset t\text{ in }\mathcal B\}.
\end{equation*}
To show that $B'$ is a block, we consider an infinite $X\subseteq\mathbb N$. Given that $B$ is a block, we find~$t(i)\in B$ that are initial segments of $X$ without its smallest $i$ elements. Note that we get~$t(i)\tl t(i+1)$. As above, we see that $X\{n\}:=t(0)\cup\ldots\cup t(n-1)$ is an initial segment of~$X$. Towards a contradiction, we assume that $g(X\{n\})\notin\three$ holds for all~$n\geq 1$. In view of $X\{n+1\}^0=X\{n\}$ we get $o(g(X\{1\}))>o(g(X\{2\}))>\ldots$, against the well foundedness of~$\Omega$. Now let $n\geq 1$ be minimal with $g(X\{n\})\in\three$. Due to the observations from the beginning of this proof, we see that $\mathcal B\ni s\sqsubset X\{n\}$ entails $s=X\{m\}$ for some~$m<n$. We thus obtain $X\{n\}\in B'$, as needed to ensure that $B'$ is a block with base~$\mathbb N$. The function $g$ restricts to a bad array $B'\to\three$, essentially as in the proof of Proposition~\ref{prop:Omega-plus-Omega-for-ACA}. To justify this claim in detail, we consider $s,t\in B'$ with~$s\tl t$. Pick an infinite $X\subseteq\mathbb N$ with $s\cup t\sqsubset X$ and write $Y$ for $X$ without its least element. We then have $s\sqsubset X$ and $t\sqsubset Y$. Define $X\{n\}$ and $Y\{n\}$ as in the proof that $B'$ is a~block. The previous considerations show that we have $s=X\{k\}$ and $t=Y\{m\}$ for some~$k,m\geq 1$. We put $n:=\max\{k,m\}$ and note $X\{n\}\tl Y\{n\}$. The point is that we have a single~$n$ with $X\{n\},Y\{n\}\in B^n$. We thus get $g(X\{n\})\not\tcleq g(Y\{n\})$ by the above. In view of $X\{k+1\}^0=X\{k\}$ and $g(X\{k\})\in\three$, we also have $g(X\{k\})=g(X\{k+1\})$. By induction this yields $g(s)=g(X\{n\})$ as well as $g(t)=g(Y\{n\})$. It follows that we have $g(s)\not\tcleq g(t)$, as required. Now if $\three$ was $\bqo$, the previous corollary would secure arithmetical comprehension and in particular weak K\H{o}nig's lemma. In the presence of the latter, the bad array $B'\to\three$ on our block~$B'$ could be transformed into a bad array $B''\to\three$ on a barrier~$B''$, by Theorem~5.12 of~\cite{cholak-RM-wpo}. But then $\three$ would not be $\bqo$ after all.
\end{proof}

As mentioned in the introduction, the following corollary strengthens a result of Marcone~\cite{marcone-survey-old} by internalizing the quantification over finite quasi orders.

\begin{corollary}[$\rca_0$]\label{cor:three-to-finite}
If $\three$ is $\bqo$, then so is any finite quasi order.
\end{corollary}
\begin{proof}
Recall the finite powerset construction from the proof of Proposition~\ref{prop:aca-from-oplus}. Given that $\three$ is $\bqo$, the same holds for $P_f(\omega\oplus\omega)$, by the previous theorem and Marcone's Theorem~5.4 in~\cite{marcone-survey-old}. We can conclude since $P_f(\omega\oplus\omega)$ contains a copy of~$\omega\times\omega$, which has antichains of any finite size. To make this more explicit, let $Q=\{q_0,\ldots,q_{n-1}\}$ be a quasi order with $n$ elements. Define $h:Q\to P_f(\omega\oplus\omega)$ by
\begin{equation*}
h(q_k):=\{\langle 0,k\rangle,\langle 1,n-k\rangle\}.
\end{equation*}
Here $\langle i,l\rangle$ refers to the element $l\in\omega$ in the $i$-th summand of $\omega\oplus\omega$. Since the sum\-mands are incomparable, $h(q_k)\leq_{P(\omega\oplus\omega)}h(q_l)$ entails both $k\leq l$ and $n-k\leq n-l$, which yields $k=l$. Hence $h$ is order reflecting, as needed to infer that $Q$ is $\bqo$.
\end{proof}

Let us write $[Y]^\omega$ for the collection of infinite subsets of~$Y$. For~$X\in[\mathbb N]^\omega$~with minimal element~$x_0$, we put $X^-:=X\backslash\{x_0\}$. A function $F:[\mathbb N]^\omega\to Q$ into a quasi order~$Q$ is called bad if there is no $X\in[\mathbb N]^\omega$ with $F(X)\leq_Q F(X^-)$. There is a straightforward correspondence between continuous functions $F:[\mathbb N]^\omega\to Q$ and arrays $f:B\to Q$ that are defined on blocks with base~$\mathbb N$ (cf.~the proof of the next corollary). Over a sufficiently strong base theory, this correspondence can be extended to the case where $F$ is Borel, as shown by Simpson~\cite{simpson-borel-bqos}. In Montalb\'an's analysis of Fra\"iss\'e's conjecture~\cite{montalban-fraisse}, an intermediate notion plays an important role: one says that $Q$ is $\delbqo$ if there is no bad $\Delta^0_2$-function~$F:[\mathbb N]^\omega\to Q$. Over the theory~$\Pi^1_1\text{-}\mathsf{CA}_0$, this is equivalent to $Q$ being $\bqo$ in terms of barriers or blocks, as shown in~\cite{montalban-fraisse}. The same reference gives a coding of $\Delta^0_2$-functions by second order objects, which is available in~$\aca_0$. When we work in~$\rca_0$, we assume that the relation $F(X)=q$ is defined by a $\Sigma^0_2$-formula, possibly with parameters. As usual, an equivalent $\Pi^0_2$-definition is provided by $\forall p\,(p\neq q\to F(X)\neq p)$. The following should dispel further worries about coding.

\begin{corollary}[$\rca_0$]
The principle of arithmetical comprehension follows from the assumption that $\three$ is $\delbqo$.
\end{corollary}
\begin{proof}
In view of Corollary~\ref{cor:three-aca}, it suffices to deduce that $\three$ is $\bqo$. We consider an array $f:B\to\three$ on a barrier~$B$. The latter can be assumed to have base~$\mathbb N$, since the bases of barriers can be formed in~$\rca_0$, due to~\cite[Lemma~1.6]{marcone-survey-old}. We now define $F:[\mathbb N]^\omega\to\three$ by setting $F(X):=f(s)$ for the unique~$s\in B$ with $s\sqsubset X$. Note that~$F$ is computable and in particular~$\Delta^0_2$. Given that $\three$ is $\delbqo$, we find a set $X\in[\mathbb N]^\omega$ with $F(X)=F(X^-)$. For the unique elements $s,t\in B$ with $s\sqsubset X$ and~$t\sqsubset X^-$, we get $s\tl t$ and $f(s)=f(t)$, as needed to confirm that $f$ is good.
\end{proof}

Let us now prove the following variant of Theorem~\ref{thm:Omega+Omega_bqo} from above.

\begin{theorem}[$\rca_0$]
If $\three$ is $\delbqo$, so is $\Omega\oplus\Omega$ for every well order~$\Omega$.
\end{theorem}
\begin{proof}
Let us assume that $\three$ is $\delbqo$, so that the previous corollary makes arithmetical comprehension available. Aiming at a contradiction, we assume that
\begin{equation*}
F_0:[\mathbb N]^\omega\to\tc\supseteq\Omega\oplus\Omega
\end{equation*}
is a bad $\Delta^0_2$-function. We want to consider the family of functions
\begin{equation*}
F_{n+1}:[\mathbb N]^\omega\to\tc\quad\text{with}\quad F_{n+1}(X):=\mathfrak x(F_n(X),F_n(X^-)),
\end{equation*}
where $\mathfrak x:\tc^2\to\tc$ is given by Lemma~\ref{lem:order-Omega-plus-Omega}. It will be important that the relation $F_n(X)=q$ is $\Delta^0_2$ in the parameters~$n,X$ and $q$. Let us write $X^{-i}$ for~$X$ without its~$i$ smallest elements, so that we have $X^{-0}=X$ and $X^{-(i+1)}=(X^{-i})^-$. It is not hard to see that there is a computable function~$\overline{\mathfrak x}:\mathbb N\to\mathbb N$ with
\begin{equation*}
F_n(X)=\overline{\mathfrak x}(s)\quad\text{for}\quad s=\langle F_0(X),\ldots,F_0(X^{-n})\rangle.
\end{equation*}
Given arithmetical comprehension, a straightforward induction proves the existence of a unique~$s$ as indicated. Since $m\in X^{-i}$ is decidable, we obtain \mbox{$\Delta^0_2$-definitions} of $F_0(X^{-i})=q$, of $s=\langle F_0(X),\ldots,F_0(X^{-n})\rangle$ and finally of \mbox{$F_n(X)=q$}. Analogous to the proof of Theorem~\ref{thm:Omega+Omega_bqo}, an induction on~$n$ shows that $F_n(X^{-i})\not\tcleq F_n(X^{-(i+1)})$ holds for all~$i$. Even though the end result is that $F_n(X)\not\tcleq F_n(X^-)$ holds fo arbitrary~$X$, it is important to have an arithmetical induction statement in which~$X$ is fixed and~$i$ varies. In view of Lemmas~\ref{lem:order-Omega-plus-Omega} and~\ref{lem:E-ordinal-descent}, we get $F_{n+1}(X)\in E(F_n(X))$ and then \mbox{$F_{n+1}(X)=F_n(X)$} when $F_n(X)\in\three$ and $o(F_{n+1}(X))<o(F_n(X))$ otherwise. Using arithmetical comprehension, we can form the function $\overline o:\mathbb N\to 1+\Omega$ with $\overline o(n)=o(F_n(X))$, for each infinite set~$X\subseteq\mathbb N$. Given that~$\Omega$ is well founded, we thus obtain a $\Delta^0_2$-definable function $\overline F:[\mathbb N]^\omega\to\three$ with $\overline F(X)=F_n(X)$ for sufficiently large~$n$. From the above it follows that~$\overline F$ is bad, which contradicts the assumption that $\three$ is $\delbqo$.
\end{proof}

In the proof of Corollary~\ref{cor:three-to-finite}, we have seen that products of $\bqo$s reduce to sums, due to an argument of Marcone~\cite{marcone-survey-old}. The analogous argument for $\delbqo$s yields the following result, which is needed in the next section. Let us recall that $\langle p,q\rangle\leq\langle p',q'\rangle$ in $P\times Q$ is equivalent to the conjunction of $p\leq_P p'$ and $q\leq_Q q'$.

\begin{corollary}[$\rca_0$]\label{cor:prod-delbqo}
If $\three$ is $\delbqo$, so is $\Omega\times\Omega$ for every well order~$\Omega$.
\end{corollary}
\begin{proof}
Recall that elements of $\Omega\oplus\Omega$ have the form $\langle i,\alpha\rangle$ with $i\in\{0,1\}$ and~$\alpha\in\Omega$. For the finite powerset construction from the proof of Proposition~\ref{prop:aca-from-oplus}, we get an order reflecting map $f:\Omega\times\Omega\to P_f(\Omega\oplus\Omega)$ by setting $f(\langle\alpha,\beta\rangle):=\{\langle 0,\alpha\rangle,\langle 1,\beta\rangle\}$. In view of the previous theorem, it is thus enough to prove that $P_f(Q)$ is $\delbqo$ whenever the same holds for~$Q$. In order to prove the contrapositive, we assume that $F:[\mathbb N]^\omega\to P_f(Q)$ is a bad $\Delta^0_2$-function. We clearly have a function $\mathfrak p:P_f(Q)^2\to Q$ with the following property: when $a\not\leq b$ holds in~$P_f(Q)$, we have $\mathfrak p(a,b)\in a$ as well as $\mathfrak p(a,b)\not\leq_Q q$ for all~$q\in b$. One readily checks that $\overline F(X):=\mathfrak p(F(X),F(X^-))$ yields a bad $\Delta^0_2$-function $\overline F:[\mathbb N]^\omega\to Q$, as needed for the contrapositive.
\end{proof}

Still over~$\rca_0$, we can conclude that every finite quasi order is $\delbqo$ if the same holds for~$\three$, as in the case of Corollary~\ref{cor:three-to-finite}.

\section{Hereditarily finite sets and trees}\label{sect:ACA-plus}

Above we have considered an order on the collection $H_{\aleph_1}(Q)$ of hereditarily countable sets with urelements from a given quasi order~$Q$. In the present section, we show that $\rca_0$ proves that the subcollection~$H_f(Q)$ of hereditarily finite sets is $\bqo$ whenever the same holds for~$Q$. We derive that $\three$~being $\bqo$ entails arithmetical recursion along~$\mathbb N$. Somewhat parallel, we then consider a collection~$T^w(Q)$ of $Q$-labelled trees with a weak embeddability relation, which has been introduced by Montalb\'an~\cite{montalban-jullien,montalban-fraisse}. As shown by the latter, $\mathsf{ATR}_0$ proves that $T^w(Q)$ is $\bqo$ when $Q$ is $\delbqo$. We show that the analogous result for the subcollection $T^w_f(Q)$ of finite trees is available over~$\rca_0$. Finally, we conclude that arithmetical transfinite recursion follows from $\three$ being $\delbqo$. Let us begin with an official definition of the hereditarily finite sets:

\begin{definition}[$\rca_0$]\label{def:H_f}
We declare that the present definition coincides with Definition~\ref{def:V(Q)} except that $H_{\aleph_1}(Q)$ is replaced by the new definiendum $H_f(Q)$ and that the set $a$ in clause~(ii) is now required to be finite.
\end{definition}

Concerning the formalization in $\rca_0$, we note that the elements of~$H_f(Q)$ are to be viewed as terms (or trees with leaf labels). By recursion over these terms, we obtain the order on $H_f(Q)$ and the (finite) supports $\supp(a)\subseteq Q$ for $a\in H_f(Q)$. We can also define a height function $\hth:H_f(Q)\to\mathbb N$ with $\hth(q):=0$ for $q\in Q$ and
\begin{equation*}
\hth(a):=\max(\{0\}\cup\{\hth(x)+1\,|\,x\in a\})\quad\text{for finite } a\subseteq H_f(Q).
\end{equation*}
To confirm that $H_f(Q)$ is a quasi order, one uses induction on $\hth(r)+\hth(s)+\hth(t)$ to check that $r\leq s$ and $s\leq t$ entail $r\leq t$. Similarly, one can~show that Lemma~\ref{lem:V(Q)-basic} holds over $\rca_0$ when $H_{\aleph_1}(Q)$ is replaced by $H_f(Q)$.

\begin{theorem}[$\rca_0$]\label{thm:H_f}
If $Q$ is $\bqo$, then so is $H_f(Q)$.
\end{theorem}
\begin{proof}
To show the contrapositive, we assume that $f:B\to H_f(Q)$ is a bad array for some barrier~$B$. From the proof of Theorem~\ref{thm:Omega+Omega_bqo}, we adopt the definitions of $B^n$ and $\mathcal B$ as well as $s^0$ and $s^1$. By recursion over sequences, we define $g:\mathcal B\to H_f(Q)$ as follows. First, we put $g(s):=f(s)$ for $s\in B^1=B$. Secondly, we set $g(s):=g(s^0)$ for $s\in\mathcal B\backslash B^1$ with $g(s^0)\in Q\subseteq H_f(Q)$ or $g(s^0)\leq g(s^1)$. Here the second disjunct is included to cover all potential cases, even though we will see that it always fails. Finally, for $s\in\mathcal B\backslash B^1$ with $g(s^0)\notin Q$ and $g(s^0)\not\leq g(s^1)$, we invoke the definition of the order on $H_f(Q)$ to pick $g(s)\in g(s^0)$ such that the following holds:
\begin{enumerate}[label=(\roman*)]
\item in case $g(s^1)\in Q$ we have $g(s)\not\leq g(s^1)$,
\item otherwise we have $g(s)\not\leq y$ for all $y\in g(s^1)$.
\end{enumerate}
By induction on~$n$, we show that $s\tl t$ entails $g(s)\not\leq g(t)$ when $s$ and $t$ lie in the same set~$B^{n+1}$. For $n=0$ it suffices to invoke the assumption that~$f$ is bad. In the induction step, we first assume that we have $g(s^0)\in Q$ and hence $g(s)=g(s^0)$. By construction we have either $g(t)=g(t^0)$ or $g(t)\in g(t^0)$, so that we get $g(t)\leq g(t^0)$ by the version of Lemma~\ref{lem:V(Q)-basic} that is available in~$\rca_0$. If we did have $g(s)\leq g(t)$, we would thus get $g(s^0)\leq g(t^0)$. The latter contradicts the induction hypothesis, since we have $s^0\tl s^1=t^0$ as in the proof of Theorem~\ref{thm:Omega+Omega_bqo}. Now assume $g(s^0)\notin Q$. The induction hypothesis ensures $g(s^0)\not\leq g(s^1)$, so that we have $g(s)\in g(s^0)$ with properties~(i) and~(ii) from above. In case we have $g(s^1)=g(t^0)\in Q$, we can invoke~(i) to obtain $g(s)\not\leq g(t^0)=g(t)$. For $g(t^0)\notin Q$, we first use the induction hy\-poth\-e\-sis to get $g(t^0)\not\leq g(t^1)$ and hence $g(t)\in g(t^0)=g(s^1)$. Then property~(ii) yields $g(s)\not\leq g(t)$, as desired. For $t\in B$ with largest element $\max(t)$, we now set
\begin{equation*}
n(t):=\max\big\{\hth(f(s))\,|\,s\in B\text{ and }s\subseteq\{0,\ldots,\max(t)\}\big\},
\end{equation*}
where $\hth:H_f(Q)\to\mathbb N$ is the height function from above. Consider
\begin{equation*}
B':=\{s(0)\cup\ldots\cup s(n)\,|\, s(0)\tl\ldots\tl s(n)\text{ with }s(i)\in B\text{ for }i\leq n=n(s(0))\}.
\end{equation*}
Based on the proof of Theorem~\ref{thm:Omega+Omega_bqo}, it is straightforward to check that $B'$ is a block with the same base as~$B$. To show that $B'$ is a barrier, we argue by contradiction. Assume that we have $s(0)\cup\ldots\cup s(m)\subset t(0)\cup\ldots\cup t(n)$ in $B'$, with the given conditions on the $s(i)$ and $t(j)$. As $B$ is a barrier, we cannot have $s(0)\subset t(0)$, so that we get $\max(s(0))\geq\max(t(0))$ and hence $m=n(s(0))\geq n(t(0))=n$.~Induction on $i\leq n$ yields $s(i)\cup\ldots\cup s(m)\subset t(i)\cup\ldots\cup t(n)$ and in particular $s(n)\subset t(n)$, against the assumption that $B$ is a barrier. For $s(0)\tl\ldots\tl s(n)$ with $s(i)\in B$, we now observe that our construction yields
\begin{alignat*}{3}
g(s(0)\cup\ldots\cup s(i+1))&={}&&g(s(0)\cup\ldots\cup s(i)) \quad &&\text{when } g(s(0)\cup\ldots\cup s(i))\in Q,\\
g(s(0)\cup\ldots\cup s(i+1))&\in{}&& g(s(0)\cup\ldots\cup s(i))\quad &&\text{otherwise}.
\end{alignat*}
As $i$ increases, the height of $g(s(0)\cup\ldots\cup s(i))$ will thus decrease until an element of~$Q$ is reached. In view of $\hth(g(s(0)))\leq n(s(0))$, it follows that $g$ restricts to a function from~$B'$ to~$Q$. As in the proof of Theorem~\ref{thm:Omega+Omega_bqo}, the latter is a bad array, which completes our proof of the contrapositive.
\end{proof}

To derive arithmetical recursion along~$\mathbb N$, we will show that a certain order $\varepsilon_\Omega$ is well founded for any well order~$\Omega$ (cf.~the introduction). The idea is to define an order reflecting map from $\varepsilon_\Omega$ into~$H_f(Q)$, for an order $Q=(\omega^2\cdot\Omega)\oplus\one$ that will be explained below. Given that $\three$ is $\bqo$, we can invoke Theorems~\ref{thm:Omega+Omega_bqo} and~\ref{thm:H_f} to conclude that the same holds for~$H_f(Q)$. To motivate the definition of $\varepsilon_\Omega$, we recall that an \mbox{$\varepsilon$-}number is a fixed point of the map $t\mapsto\omega^t$ on the ordinals. Any ordinal~$t$ that is not an $\varepsilon$-number has a unique Cantor normal form $t=\omega^{t_0}+\ldots+\omega^{t_{n-1}}$ with $t_{n-1}\preceq\ldots\preceq t_0\prec t$. The following definition captures this structure.

\begin{definition}[$\rca_0$]\label{def:eps_Omega}
Given a linear order~$\Omega$, we use simultaneous recursion to define a set $\varepsilon_\Omega$ of terms and a binary relation $\prec$ on this set. The terms in $\varepsilon_\Omega$ are generated by the following clauses:
\begin{enumerate}[label=(\roman*)]
\item for each $\alpha\in\Omega$ we have a term $\varepsilon_\alpha\in\varepsilon_\Omega$,
\item we get $\langle t_0,\ldots,t_{n-1}\rangle\in\varepsilon_\Omega$ for any previously constructed $t_i\in\varepsilon_\Omega$ such that
\begin{itemize}
\item either $n=1$ and $t_0$ is not of the form $\varepsilon_\alpha$,
\item or $n\in\{0\}\cup\{2,3,\ldots\}$ and $t_i\preceq t_{i-1}$ for $0<i<n$, where $s\preceq t$ denotes that we have $s\prec t$ or $s$ and $t$ are the same term.
\end{itemize}
\end{enumerate}
In order to define our binary relation, we stipulate that $s\prec t$ holds precisely when one of the following clauses applies:
\begin{enumerate}[label=(\roman*')]
\item $s=\varepsilon_\alpha$ and $t=\varepsilon_\beta$ with $\alpha<\beta$ in~$\Omega$,
\item $s=\varepsilon_\alpha$ and $t=\langle t_0,\ldots,t_{n-1}\rangle$ with $n>0$ and $s\preceq t_0$,
\item $s=\langle s_0,\ldots,s_{m-1}\rangle$ and $t=\varepsilon_\beta$ with $m=0$ or $s_0\prec t$,
\item $s=\langle s_0,\ldots,s_{m-1}\rangle$ and $t=\langle t_0,\ldots,t_{n-1}\rangle$ with
\begin{itemize}
\item either $m<n$ and $s_i=t_i$ for all $i<m$,
\item or there is $j<\min(m,n)$ with $s_j\prec t_j$ and $s_i=t_i$ for all $i<j$.
\end{itemize}
\end{enumerate}
\end{definition}

To justify the recursion, we consider the length function $\len:\varepsilon_\Omega\to\mathbb N$~with
\begin{equation*}
\len(\varepsilon_\alpha):=0\quad\text{and}\quad\len(\langle t_0,\ldots,t_{n-1}\rangle):=1+\textstyle\sum_{i<n}\len(t_i).
\end{equation*}
One can decide $t\in\varepsilon_\Omega$ and $r\prec s$ by simultaneous recursion on $\len(t)$ and $\len(r)+\len(s)$, respectively, where the recursive calls do not lead out of a finite set of subterms. For an inductive proof that $\prec$ is a linear order on $\varepsilon_\Omega$, one can also employ the given length function (see, e.\,g., Lemma~2.3 of~\cite{rathjen-afshari}). To decide $s\prec t$ one may first look at the largest subterm of the form $\varepsilon_\alpha$ and then at the number of exponentials that are applied to it. More precisely, these invariants are defined as follows.

\begin{definition}[$\rca_0$]
Consider a linear order $\Omega$ with a smallest element that is denoted by~$0$. We define functions $e:\varepsilon_\Omega\to\Omega$ and $d:\varepsilon_\Omega\to\mathbb N$ recursively by
\begin{alignat*}{3}
e(\varepsilon_\alpha)&:=\alpha,&\qquad e(\langle\rangle)&:=0,&\qquad e(\langle t_0,\ldots,t_n\rangle)&:=e(t_0),\\
d(\varepsilon_\alpha)&:=0,&\qquad d(\langle\rangle)&:=0,&\qquad d(\langle t_0,\ldots,t_n\rangle)&:=d(t_0)+1.
\end{alignat*}
\end{definition}

As indicated above, our invariants have the following relation with the order.

\begin{lemma}[$\rca_0$]\label{lem:exp-depth}
Assume that we have $s\preceq t$ in $\varepsilon_\Omega$, for some $\Omega$ with a smallest element. We then get $e(s)\leq_\Omega e(t)$. If the latter is an equality, we get $d(s)\leq d(t)$.
\end{lemma}
\begin{proof}
By induction on~$s$, one can check that $\beta\leq_\Omega e(s)$ and $s\neq\langle\rangle$ entail $\varepsilon_\beta\preceq s$. Based on this fact, we can verify the claim by induction on~$\len(s)+\len(t)$. The only interesting case is clause~(iii') of Definition~\ref{def:eps_Omega}. Given $s_0\prec t=\varepsilon_\beta$, we inductively get $e(s)=e(s_0)\leq_\Omega e(t)=\beta$. If this was an equality, the above would yield $t\preceq s$ and hence $s=t$, which is impossible in clause~(iii').
\end{proof}

Let us write $\omega^2\cdot\Omega$ for the linear order with underlying set $\Omega\times\mathbb N^2$ and
\begin{equation*}
\langle\alpha,m,n\rangle<_{\omega^2\cdot\Omega}\langle\alpha',m',n'\rangle\quad\Leftrightarrow\quad\begin{cases}
\text{either }\alpha<_\Omega\alpha',\\
\text{or }\alpha=\alpha'\text{ and }m<m',\\
\text{or }\alpha=\alpha'\text{ and }m=m'\text{ and }n<n'.
\end{cases}
\end{equation*}
Given $\langle\alpha_0,m_0,n_0\rangle\geq\langle\alpha_1,m_1,n_1\rangle\geq\ldots$ in $\omega^2\cdot\Omega$, we obtain $\alpha_0\geq\alpha_1\geq\ldots$ in~$\Omega$. When the latter is a well order, we thus find an $N\in\mathbb N$ with $\alpha_i=\alpha_N$ for all~$i\geq N$. Similarly, we then learn that the $m_i$ and the $n_i$ must eventually become constant. This shows that $\omega^2\cdot\Omega$ is a well order if the same holds for~$\Omega$, provably in $\rca_0$. In terms of the given order, the lemma above asserts
\begin{equation*}
s\preceq t\quad\Rightarrow\quad\langle e(s),d(s),0\rangle\leq_{\omega^2\cdot\Omega}\langle e(t),d(t),0\rangle.
\end{equation*}
We write $\one$ for the order with a single element, which we denote~by~$\star$. According to the previous section, the partial order $(\omega^2\cdot\Omega)\oplus\one$ with incomparable summands has underlying set $\{\langle 0,q\rangle\,|\,q\in\omega^2\cdot\Omega\}\cup\{\langle 1,\star\rangle\}$. To improve readability, we will write $\overline q$ and $\star$ at the place of $\langle 0,q\rangle$ and $\langle 1,\star\rangle$, where $q$ has the form $\langle\alpha,m,n\rangle\in\omega^2\cdot\Omega$.

\begin{definition}[$\rca_0$]
For a linear order~$\Omega$ with a minimal element, we define a function $\mathfrak j:\varepsilon_\Omega\to H_f((\omega^2\cdot\Omega)\oplus\one)$ recursively by $\mathfrak j(\varepsilon_\alpha):=\{\overline{\langle\alpha,0,0\rangle},\star\}$ and
\begin{equation*}
\mathfrak j(t):=\{\star\}\cup\left\{\left.\left\{\overline{\langle e(t),d(t),i\rangle},\mathfrak j(t_i)\right\}\,\right|\,i<n\right\}\quad\text{for}\quad t=\langle t_0,\ldots,t_{n-1}\rangle.
\end{equation*}
\end{definition}

The following is readily verified by induction on~$t\in\varepsilon_\Omega$.

\begin{lemma}[$\rca_0$]\label{lem:j-exp-depth}
Given $\overline{\langle\alpha,m,0\rangle}\leq\mathfrak j(t)$, we get $\langle\alpha,m,0\rangle\leq\langle e(t),d(t),0\rangle$.
\end{lemma}

We now derive the result that was promised above.

\begin{theorem}[$\rca_0$]
The function $\mathfrak j:\varepsilon_\Omega\to H_f((\omega^2\cdot\Omega)\oplus\one)$ is order reflecting, for any linear order~$\Omega$ that has a minimal element.
\end{theorem}
\begin{proof}
To show that $\mathfrak j(s)\leq\mathfrak j(t)$ entails $s\preceq t$, we argue by induction on~$\len(s)+\len(t)$ and distinguish cases according to the clauses from Definition~\ref{def:eps_Omega}. For the induction step, we assume $\mathfrak j(s)\leq\mathfrak j(t)$. If we have $s=\varepsilon_\alpha$ and $t=\varepsilon_\beta$, the claim is immediate. Now assume we have $s=\varepsilon_\alpha$ and $t=\langle t_0,\ldots,t_{n-1}\rangle$. In view of $\overline{\langle\alpha,0,0\rangle}\not\leq\star$, there must be an $i<n$ with
\begin{equation*}
\overline{\langle\alpha,0,0\rangle}\leq\left\{\overline{\langle e(t),d(t),i\rangle},\mathfrak j(t_i)\right\}.
\end{equation*}
If this inequality is due to $\langle\alpha,0,0\rangle\leq\langle e(t),d(t),i\rangle$, then we have $\alpha\leq e(t)$, so that we get $s=\varepsilon_\alpha\preceq t$ as in the proof of Lemma~\ref{lem:exp-depth}. Otherwise we have $\overline{\langle\alpha,0,0\rangle}\leq\mathfrak j(t_i)$. As Definition~\ref{def:eps_Omega} ensures $t_i\preceq t_0$, we can use Lemmas~\ref{lem:exp-depth} and~\ref{lem:j-exp-depth} to infer
\begin{equation*}
\langle e(s),d(s),0\rangle=\langle\alpha,0,0\rangle\leq\langle e(t_i),d(t_i),0\rangle\leq\langle e(t_0),d(t_0),0\rangle<\langle e(t),d(t),0\rangle.
\end{equation*}
We now get $s\prec t$ by Lemma~\ref{lem:exp-depth}, as the order~$\prec$ on $\varepsilon_\Omega$ is linear. Let us continue with terms of the form $s=\langle s_0,\ldots,s_{m-1}\rangle$ and $t=\varepsilon_\beta$. We may assume $m>0$, for otherwise $s\prec t$ is immediate. Note that we have $\star\in\supp(\mathfrak j(s_0))$ and hence
\begin{equation*}
\left\{\overline{\langle e(s),d(s),0\rangle},\star\right\}\subseteq\supp(x)\quad\text{for}\quad x:=\left\{\overline{\langle e(s),d(s),0\rangle},\mathfrak j(s_0)\right\}\in\mathfrak j(s).
\end{equation*}
Due to Lemma~\ref{lem:V(Q)-basic} (see also the paragraph after Definition~\ref{def:H_f}), we get $x\not\leq\overline{\langle\beta,0,0\rangle}$ as well as $x\not\leq\star$. This entails $\mathfrak j(s)\not\leq\mathfrak j(\varepsilon_\beta)=\mathfrak j(t)$, against our assumption. Finally, let $s=\langle s_0,\ldots,s_{m-1}\rangle$ and $t=\langle t_0,\ldots,t_{n-1}\rangle$. For each $i<m$ we find $k<n$ with
\begin{equation*}
\left\{\overline{\langle e(s),d(s),i\rangle},\mathfrak j(s_i)\right\}\leq\left\{\overline{\langle e(t),d(t),k\rangle},\mathfrak j(t_k)\right\}.
\end{equation*}
If we have $\overline{\langle e(s),d(s),i\rangle}\leq\mathfrak j(t_k)$, then we get $s\prec t$ via Lemmas~\ref{lem:exp-depth} and~\ref{lem:j-exp-depth} as before. So now assume $\langle e(s),d(s),i\rangle\leq\langle e(t),d(t),k\rangle$. In particular, we obtain $e(s)\leq e(t)$. If this inequality is strict, we again get $s\prec t$ due to Lemma~\ref{lem:j-exp-depth}. We may thus assume $e(s)=e(t)$ and then also $d(s)=d(t)$, for the same reason. The point is that we then get $i\leq k$. We can now adopt an argument of Marcone~\cite{marcone-survey-old}, which was already presented in the proof of Proposition~\ref{prop:aca-from-oplus}. To make this explicit, we first note that $\star\in\supp(\mathfrak j(s_i))$ entails $\mathfrak j(s_i)\not\leq\overline{\langle e(t),d(t),k\rangle}$. We must thus have $\mathfrak j(s_i)\leq\mathfrak j(t_k)$, which inductively yields $s_i\preceq t_k\preceq t_i$. Since this holds for any $i<m\leq n$, we can conclude $s\preceq t$ by clause~(iv') of Definition~\ref{def:eps_Omega}.
\end{proof}

The following bound on $\three$ being~$\bqo$, which improves the one from Corollar~\ref{cor:three-aca}, is one of the main results of the present paper. We note that arithmetical recursion along~$\mathbb N$ is equivalent to the existence of~$\omega$-jumps.

\begin{corollary}[$\rca_0$]\label{cor:three-to-acaplus}
The statement that $\three$ is $\bqo$ entails arithmetical recursion along the order~$\mathbb N$ (which is the central axiom of $\aca_0^+$).
\end{corollary}
\begin{proof}
Arithmetical recursion along~$\mathbb N$ is equivalent to the statement that $\varepsilon_\Omega$ is well founded for any well order~$\Omega$, as shown by Marcone and Montalb\'an~\cite{marcone-montalban} as well as Afshari and Rathjen~\cite{rathjen-afshari} (by different methods from computability and proof theory). As $\varepsilon_0$ embeds into $\varepsilon_1$, it suffices to consider a well order $\Omega$ that is non-empty. The latter will thus have a smallest element, so that the previous considerations apply. Given that $\three$ is $\bqo$, we can infer that the same holds for $H_f((\omega^2\cdot\Omega)\oplus\one)$, by Theorems~\ref{thm:Omega+Omega_bqo} and~\ref{thm:H_f}. Due to the previous theorem, it follows that $\varepsilon_\Omega$ is well founded, as required.
\end{proof}

In the following, we prove a lower bound on the strength of the statement that the order $\three$ is $\delbqo$. For this purpose, we consider the finite $Q$-labelled trees with two order relations that were introduced by Montalb\'an~\cite{montalban-jullien,montalban-fraisse}.

\begin{definition}
For a quasi order~$Q$, let $T_f(Q)$ be the collection of finite ordered trees with vertex labels from~$Q$, as generated by the following recursive clause:
\begin{itemize}
\item given any $q\in Q$ and previously constructed $b_i\in T_f(Q)$ for $i<n$ (possibly with $n=0$), we add a new element $q\star\langle b_0,\ldots,b_{n-1}\rangle\in T_f(Q)$.
\end{itemize}
For $a=p\star\langle a_0,\ldots,a_{m-1}\rangle$ and $b=q\star\langle b_0,\ldots,b_{n-1}\rangle$ in $T_f(Q)$, we recursively declare
\begin{alignat*}{3}
a&\leq^s{}&& b\quad&&\Leftrightarrow\quad\begin{cases}
\text{$p\leq_Q q$ and for each $i<m$ there is $j<n$ with $a_i\leq^s b_j$},\\
\text{or $a\leq^s b_j$ for some~$j<n$},
\end{cases}\\
a&\leq^w{} &&b\quad&&\Leftrightarrow\quad\begin{cases}
\text{$p\leq_Q q$ and $a_i\leq^w b$ for all $i<m$},\\
\text{or $a\leq^w b_j$ for some~$j<n$}.
\end{cases}
\end{alignat*}
We write $T_f^s(Q)$ and $T_f^w(Q)$ for $T_f(Q)$ with the relation $\leq^s$ and $\leq^w$, respectively.
\end{definition}

A straightforward induction shows that we have $a\leq^s b$ or $a\leq^w b$, respectively, if any only if there is a homomorphism or weak homo\-morphism $a\to b$ in the sense of Montalb\'an~\cite{montalban-fraisse}. Let us point out that homomorphisms map proper subtrees into proper subtrees but need not be injective. In our definition, this corresponds to the fact that the same $j<n$ may witness $a_i\leq^s b_j$ for different~$i<m$. Further inductions confirm that $\leq^s$ and $\leq^w$ are quasi orders. In particular, we can invoke reflexivity to get $b_j\leq^s q\star\langle b_0,\ldots,b_{n-1}\rangle$ for $j<n$. Hence the condition $a_i\leq^s b_j$ is stronger than $a_i\leq^s b$, which inductively shows that $a\leq^s b$ implies $a\leq^w b$. In the converse direction, we have the following result of Montalb\'an (which was originally proved for the more general case of infinite trees). We recall that $Q\oplus\one$ denotes the extension of $Q$ by a new element that is incomparable with any other.

\begin{lemma}[$\rca_0$; \cite{montalban-fraisse}]\label{lem:strong-weak-trees}
There is an order embedding of $T_f^s(Q)$ into $T_f^w(Q\oplus\one)$, for any quasi order~$Q$.
\end{lemma}

We will combine the lemma with the following proposition, which follows from work of H.~Friedman, A.~Montalb\'an and A.~Weiermann. This work builds on a previous result of Friedman, which characterizes~$\mathsf{ATR}_0$ in terms of orders~$\varphi_\Omega 0$ that are related to the Veblen hierarchy. The proposition is derived by showing that $\varphi_\Omega0$ admits an order reflecting map into~$T_f^s(\Omega\oplus\mathbf 2)$, where $\mathbf 2$ is the discrete order with two elements. We note that Marcone and Montalb\'an~\cite{marcone-montalban-hausdorff} have constructed an order reflecting function from $\varphi_20$ into~$T_f^s(\mathbf 2)$.

\begin{proposition}[$\rca_0$; \cite{friedman-montalban-weiermann}]\label{prop:atr-montalban-trees}
Arithmetical transfinite recursion follows from the statement that $T_f^s(\Omega\oplus\mathbf 2)$ is a well quasi order for any well order~$\Omega$.
\end{proposition}

To obtain a bound on the statement that $\three$ is $\delbqo$, we will combine the cited results with the following theorem. As noted above, Montalb\'an~\cite{montalban-fraisse} has shown that the corresponding theorem for infinite trees holds over $\mathsf{ATR}_0$.

\begin{theorem}[$\rca_0$]
If $Q$ is~$\delbqo$, then $T_f^w(Q)$ is~$\bqo$.
\end{theorem}
\begin{proof}
Aiming at the contrapositive, we assume that $f:B\to T_f^w(Q)$ is a bad array on some barrier~$B$. We adopt the definitions of $B^n$ and $\mathcal B$ as well as $s^0$ and $s^1$ from the proof of Theorem~\ref{thm:Omega+Omega_bqo}. In order to define a function $g:\mathcal B\to T_f^w(Q)$, we first put $g(s):=f(s)$ for $s\in B=B^1$. For a given $s\in B^{n+1}$ with $n>0$, we write $g(s^0)=q\star\langle a_0,\ldots,a_{n-1}\rangle$ for the value that is provided by recursion. We then put
\begin{equation*}
g(s)=\begin{cases}
a_{\min(I)} & \text{when $I:=\{i<n\,|\,a_i\not\leq^w g(s^1)\}$ is non-empty},\\
g(s^0) & \text{otherwise}.
\end{cases}
\end{equation*}
Let us record the following consequences of the definition:
\begin{enumerate}[label=(\roman*)]
\item we have $g(s)\leq^w g(s^0)$ for any $s\in\mathcal B\backslash B^1$,
\item given $g(s)=g(s^0)=q\star\langle a_0,\ldots,a_{n-1}\rangle$, we get $a_i\leq^w g(s^1)$ for all~$i<n$.
\end{enumerate}
We show that $s\tl t$ entails $g(s)\not\leq^w g(t)$ when we have $s,t\in B^n$ for a single~$n>0$. The case of $n=1$ is covered by the assumption that~$f$ is bad. For $n>1$, we distinguish cases according to the definition of~$g(s)$. If the latter is given by the first case above, then we get $g(s)\not\leq^w g(s^1)=g(t^0)$, where we have $s^1=t^0$ as in the proof of Theorem~\ref{thm:Omega+Omega_bqo}. To get $g(s)\not\leq^w g(t)$, we use~(i) with $t$ at the place~of~$s$. In the remaining case we have~$g(s)=g(s^0)$. Due to $s^0\tl s^1$, the induction~hypoth\-esis yields $g(s)\not\leq^w g(s^1)$, so that we can conclude as before. Let us now consider
\begin{equation*}
h:\mathcal B\to Q\quad\text{with}\quad h(s):=q\quad\text{for}\quad g(s)=q\star\langle a_0,\ldots,a_{n-1}\rangle.
\end{equation*}
As we have $g(s^0)\not\leq^w g(s^1)$ for $s\in\mathcal B\backslash B^1$, we get the following from~(ii) above:
\begin{enumerate}[label=(\roman*)]\setcounter{enumi}{2}
\item if we have $g(s)=g(s^0)$, then we obtain $h(s)\not\leq_Q h(s^1)$ .
\end{enumerate}
For any~$X\in[\mathbb N]^\omega$ and $n>0$, we have a unique $X\{n\}\sqsubset X$ with $X\{n\}\in B^n$ , as in the proof of Theorem~\ref{thm:Omega+Omega_bqo}. Let us recall that we get $X\{n+1\}^0=X\{n\}$ and similarly $X\{n+1\}^1=X^-\{n\}$. We claim that the $\Sigma^0_2$-condition
\begin{equation*}
F(X)=q\quad:\Leftrightarrow\quad\exists n>0\forall m\geq n.\,h(X\{m\})=q
\end{equation*}
defines the graph of a total function $F:[\mathbb N]^\omega\to Q$. Indeed, if $F(X)$ was undefined, then $h(X\{m\})$ and hence $g(X\{m\})$ would not stabilize. This would yield a strictly increasing sequence of indices $m_0<m_1<\ldots$ with
\begin{equation*}
g(X\{m_{i-1}\})=\ldots=g(X\{m_i-1\})\neq g(X\{m_i\}).
\end{equation*}
In view of $X\{m_i-1\}=X\{m_i\}^0$ and the definition of~$g$, the inequality would mean that $g(X\{m_i\})$ is a proper subtree of $g(X\{m_{i-1}\})$, which must fail for some~$i>0$. For sufficiently large~$m$, we can now use~(iii) to get
\begin{equation*}
F(X)=h(X\{m\})\not\leq_Q h(X\{m\}^1)=h(X^-\{m-1\})=F(X^-),
\end{equation*}
as needed to show that~$F$ is bad.
\end{proof}

To conclude, we derive the second main result of this paper.

\begin{corollary}[$\rca_0$]
The statement that $\three$ is $\delbqo$ entails arithmetical trans\-finite recursion (which is the central axiom of~$\mathsf{ATR}_0$).
\end{corollary}
\begin{proof}
Due to Lemma~\ref{lem:strong-weak-trees} and Proposition~\ref{prop:atr-montalban-trees}, it suffices to derive that $T_f^w(\Omega\oplus\three)$ is a well partial order for any well order~$\Omega$. In view of the previous theorem, this reduces to the claim that~$\Omega\oplus\three$ is $\delbqo$. It is enough to show that the same holds for $\Omega\times\Omega$, as we may assume that~$\Omega$ has at least four elements. By Corollary~\ref{cor:prod-delbqo}, we can finally reduce to the assumption that $\three$ is $\delbqo$.
\end{proof}

\bibliographystyle{amsplain}
\bibliography{3-bqo}

\newcommand{\noopsort}[1]{}
\providecommand{\bysame}{\leavevmode\hbox to3em{\hrulefill}\thinspace}
\providecommand{\MR}{\relax\ifhmode\unskip\space\fi MR }
% \MRhref is called by the amsart/book/proc definition of \MR.
\providecommand{\MRhref}[2]{%
  \href{http://www.ams.org/mathscinet-getitem?mr=#1}{#2}
}
\providecommand{\href}[2]{#2}
\begin{thebibliography}{10}

\bibitem{rathjen-afshari}
Bahareh Afshari and Michael Rathjen, \emph{Reverse mathematics and
  well-ordering principles: {A} pilot study}, Annals of Pure and Applied Logic
  \textbf{160} (2009), 231--237.

\bibitem{cholak-RM-wpo}
Peter Cholak, Alberto Marcone, and Reed Solomon, \emph{Reverse mathematics and
  the equivalence of definitions for well and better quasi-orders}, The Journal
  of Symbolic Logic \textbf{69} (2004), no.~3, 683--712.

\bibitem{forster-bqo}
Thomas Forster, \emph{Better-quasi-orderings and coinduction}, Theoretical
  Computer Science \textbf{309} (2003), 111--123.

\bibitem{freund-higman-bqo}
Anton Freund, \emph{Higman's lemma is stronger for better quasi orders}, 2022,
  preprint available as \texttt{arXiv:2205.04336}.

\bibitem{friedman-montalban-weiermann}
Harvey Friedman, Antonio Montalb\'an, and Andreas Weiermann, \emph{Phi
  function}, 2007, draft.

\bibitem{girard87}
Jean-Yves Girard, \emph{Proof theory and logical complexity, volume 1}, Studies
  in Proof Theory, Bibliopolis, Napoli, 1987.

\bibitem{hirst94}
Jeffry~L. Hirst, \emph{Reverse mathematics and ordinal exponentiation}, Annals
  of Pure and Applied Logic \textbf{66} (1994), 1--18.

\bibitem{kruskal-rediscovered}
Joseph Kruskal, \emph{The theory of well-quasi-ordering: {A} frequently
  discovered concept}, Journal of Combinatorial Theory, Series A \textbf{13}
  (1972), no.~3, 297--305.

\bibitem{laver71}
Richard Laver, \emph{On {F}ra{\"i}ss{\'e}'s order type conjecture}, Annals of
  Mathematics \textbf{93} (1971), no.~1, 89--111.

\bibitem{marcone-survey-old}
Alberto Marcone, \emph{{WQO} and {BQO} theory in subsystems of second order
  arithmetic}, Reverse Mathematics 2001 (Stephen Simpson, ed.), Lecture Notes
  in Logic, vol.~21, Cambridge University Press, 2005, pp.~303--330.

\bibitem{marcone-survey-new}
\bysame, \emph{The reverse mathematics of wqos and bqos}, Well-Quasi Orders in
  Computation, Logic, Language and Reasoning (Peter Schuster, Monika
  Seisenberger, and Andreas Weiermann, eds.), Trends in Logic, vol.~53,
  Springer, Cham, 2020, pp.~189--219.

\bibitem{marcone-montalban-hausdorff}
Alberto Marcone and Antonio Montalb{\'a}n, \emph{On {F}ra\"iss\'e's conjecture
  for linear orders of finite {H}ausdorff rank}, Annals of Pure and Applied
  Logic \textbf{160} (2009), no.~3, 355--367.

\bibitem{marcone-montalban}
\bysame, \emph{The {V}eblen functions for computability theorists}, The Journal
  of Symbolic Logic \textbf{76} (2011), 575--602.

\bibitem{montalban-jullien}
Antonio Montalb{\'a}n, \emph{Equivalence between {F}ra\"iss\'e's conjecture and
  {J}ullien's theorem}, Annals of Pure and Applied Logic \textbf{139} (2006),
  1--42.

\bibitem{montalban-open-problems}
\bysame, \emph{Open questions in reverse mathematics}, The Bulletin of Symbolic
  Logic \textbf{17} (2011), 431--454.

\bibitem{montalban-fraisse}
\bysame, \emph{Fra{\"i}ss{\'e}'s conjecture in {$\Pi^1_1$}-comprehension},
  Journal of Mathematical Logic \textbf{17} (2017), no.~2, article no.~1750006.

\bibitem{nash-williams-trees}
Crispin~{St.\ J.\ A.} Nash-Williams, \emph{On well-quasi-ordering infinite
  trees}, Mathematical Proceedings of the Cambridge Philosophical Society
  \textbf{61} (1965), 697--720.

\bibitem{nash-williams-bqo}
\bysame, \emph{On better-quasi-ordering transfinite sequences}, Mathematical
  Proceedings of the Cambridge Philosophical Society \textbf{64} (1968),
  273--290.

\bibitem{pequignot-survey}
Yann Pequignot, \emph{Towards better: {A} motivated introduction to
  better-quasi-orders}, EMS Surveys in Mathematical Sciences \textbf{4} (2017),
  no.~2, 185--218.

\bibitem{rathjen-weiermann-atr}
Michael Rathjen and Andreas Weiermann, \emph{Reverse mathematics and
  well-ordering principles}, Computability in Context: Computation and Logic in
  the Real World (S.~Barry Cooper and Andrea Sorbi, eds.), Imperial College
  Press, 2011, pp.~351--370.

\bibitem{robertson-seymour-polytime}
Neil Robertson and Paul Seymour, \emph{Graph minors.~{XIII}.\ {T}he disjoint
  paths problem}, Journal of Combinatorial Theory, Series B \textbf{63} (1995),
  no.~1, 65--110.

\bibitem{robertson-seymour-gm}
\bysame, \emph{Graph minors. {XX}. {W}agner's conjecture}, Journal of
  Combinatorial Theory, Series B \textbf{92} (2004), no.~2, 325--357.

\bibitem{simpson-borel-bqos}
Stephen Simpson, \emph{Bqo theory and {F}ra\"iss\'e's conjecture},
  \normalfont{chapter in the book} `Recursive Aspects of Descriptive Set
  Theory' \normalfont{by R.~Mansfield and G.~Weitkamp}, Oxford University
  Press, 1985, pp.~124--138.

\bibitem{simpson09}
\bysame, \emph{Subsystems of second order arithmetic}, Perspectives in Logic,
  Cambridge University Press, 2009.

\end{thebibliography}

\end{document}